\documentclass[a4paper,11pt]{article}
\usepackage{amssymb}
\usepackage{amsfonts}
\usepackage{amsmath}
\usepackage{amsthm}

\usepackage[colorlinks=true]{hyperref}

\usepackage{url}
\def\i{\sqrt{-1}}

\def\Tr{\mathrm{tr}}
\def\Vol{\mathrm{Vol}}
\def\rank{\mathrm{rank}}
\def\rk{\mathrm{rk}}

\def\ddbar{\sqrt{-1}\partial\bar{\partial}}
\def\L{\sqrt{-1}\Lambda_\omega\bar{\partial}\partial}
\def\dvol{\frac{\omega^b}{b!}}
\def\Ker{\mathrm{Ker}}
\def\Todd{\mathrm{Todd}}
\def\Id{\mathrm{Id}}
\def\Scal{\mathrm{scal}}
\def\End{\mathrm{End}}
\def\Met{\mathrm{Met}}
\def\rk{\mathrm{rk}}
\def\PP{\mathbb{P}}
\def\Aut{\mathrm{Aut}}
\def\ch{\mathrm{ch}}
\newtheorem{theorem}{Theorem}
\newtheorem{lemma}[theorem]{Lemma}
\newtheorem{corollary}[theorem]{Corollary}
\newtheorem{proposition}[theorem]{Proposition}

\theoremstyle{definition}
\newtheorem{definition}[theorem]{Definition}
\newtheorem{remark}[theorem]{Remark}

\title{A note on Chow stability of the Projectivisation of Gieseker Stable Bundles}
\author{Julien Keller and Julius Ross}
\begin{document}
\maketitle

\begin{abstract}
We investigate Chow stability of projective bundles $\mathbb P(E)$ where $E$ is a strictly Gieseker stable bundle over a base manifold that has constant scalar curvature.   We show that, for suitable polarisations $\mathcal L$, the pair $(\mathbb P(E),\mathcal L)$ is Chow stable and give examples for which it is not asymptotically Chow stable.
\end{abstract}

\section{Introduction}

Our aim is to investigate the connection between stability of a vector bundle $E$ and stability of the projective bundle $\PP(E)$ as a polarised manifold.   Roughly speaking one expects that $\PP(E)$ is stable, with respect to polarisations that make the fibres sufficiently small, if and only if $E$ is a stable vector bundle over a base that is stable as a manifold.

The first result along these lines is due to I.\ Morrison \cite{M} who showed that if $E$ is a stable rank $2$ bundle on a smooth Riemann surface $B$ then the ruled surface $\pi\colon \PP(E)\to B$ is Chow stable with respect to the polarisation $L_k=\mathcal O_{\PP(E)}(1)\otimes\pi^* \mathcal O_B(k)$ for $k\gg 0$.   Later, building on the work of E.\ Calabi, A.\ Fujiki, C.\ Lebrun and many others, V.\ Apostolov, D.\ Calderbank, P.\ Gauduchon and C.\ T\o nnesen-Friedman have provided a complete understanding of the situation for higher rank bundles over a smooth Riemann surface.   They show there is a constant scalar curvature K\"ahler metric (cscK in short) in any K\"ahler class on $\PP(E)$ if and only if the bundle $E$ is Mumford polystable \cite{ACGT,ACGT2,AT1}.  Such metrics are related to stability through the Yau-Tian-Donaldson conjecture (see, for example, \cite{RT} for an account).  In particular it implies through work of S.K.\ Donaldson \cite{D1} that $\PP(E)$ is asymptotically Chow stability, by which one means that if $r$ is sufficiently large then the embedding of $\PP(E)$ into projective space using the linear series determined by $L_k^r$ is Chow stable (see also A.\ Della Vedova and F.\ Zuddas \cite{DV1} for a generalisation).

There are at least two extensions to the case when the base $B$ has higher dimension.  First, a result of Y.-J.\ Hong \cite{H} states if $E$ is a Mumford-stable bundle of any rank over a smooth base $B$ that has a discrete automorphism group and a cscK metric, then $\PP(E)$ will also admit a cscK metric, again making the fibres small.  (Once again, from \cite{D1},  this implies that $\PP(E)$ is asymptotically Chow stable.)  Second, a result of R.\ Seyyedali states that in fact under these conditions, $\PP(E)$ is Chow stable with respect to $L_k$ for $k\gg 0$, the novelty here being that stability is not taken asymptotically, which implies Morrison's result. \smallskip

The purpose of this note is to relax the assumption that $E$ is Mumford stable and instead consider bundles that are merely Gieseker stable.  To state the theorems precisely, let $B$ be a smooth polarised manifold carrying an ample line bundle $L$ such that the automorphism group $\Aut(B,L)/\mathbb C^*$ is discrete.   The projective bundle $\pi\colon \mathbb P(E)\to B$ carries a tautological bundle $\mathcal{O}_{\mathbb{P}E}(1)$, and the line bundle 
$$\mathcal{L}_k:=\mathcal{O}_{\mathbb{P}E}(1)\otimes \pi ^* L ^{k}$$
is ample for $k$ sufficiently large.

\begin{theorem}\label{thm:mainstable}
Suppose that $E$ is Gieseker stable and its Jordan-H\"older filtration is given by subbundles, and assume there is a constant scalar curvature K\"ahler metric in the class $c_1(L)$.  Then $(\mathbb P(E),\mathcal L_k)$ is Chow stable for $k$ sufficiently large.
\end{theorem}

\begin{theorem}\label{notasymptoticchowstable}
  Suppose that $E$ is a rank 2 bundle over a surface and $F$ is a subbundle of $E$ such that $E/F$ is locally free.  Suppose furthermore $\mu(F)=\mu(E)$ and
\[ 4\, \left(\ch_2(E)/2- \ch_2 (F) \right) + c_{1}(B)\left( c_1(E)/2-c_{1}(F)  \right)<0, \]
where $\ch_2$ denotes the degree $2$ term in the Chern character. Then for $k$ sufficiently large, $(\mathbb{P}(E),\mathcal{L}_k)$ is not K-semistable and thus not asymptotic Chow stable.
\end{theorem}

These theorems should be compared to an observation of  D.\ Mumford that a quartic cuspidal plane curve is Chow stable (as a plane curve), but not asymptotically Chow stable \cite[Section 3]{Mum}.  We consider the results here noteworthy insofar as it gives smooth examples of the same nature (see Section \ref{ex}).\smallskip

When $E$ is Mumford stable Theorem \ref{thm:mainstable} is due to Seyyedali \cite{S} which in turn builds on work of Donaldson \cite{D1}.   Our proof will be along the same lines,  the main innovation being to replace the Hermitian-Einstein metrics used by  Seyyedali with the almost Hermitian-Einstein metrics on a Gieseker stable bundle furnished by N.C.\ Leung \cite{L}. Under the above assumptions, we construct a sequence of metrics that are balanced i.e make the Bergman function for $(\PP(E),\mathcal{L}_k)$ constant. The proof of Theorem \ref{notasymptoticchowstable} consists of a calculation of the Futaki invariant similar to that of Ross-Thomas \cite{RT}.\smallskip

\noindent {\bf Acknowledgements}: This research was supported by a Marie Curie International Reintegration Grant within the 7\textsuperscript{th} European Community Framework Programme and the  French Agence Nationale de la Recherche - ANR project MNGNK (ANR-10-BLAN-0118).\\

\noindent {\bf Conventions}: If $\pi\colon E\to B$ is a vector bundle then $\pi\colon \PP(E)\to B$ shall denote the space of complex \emph{hyperplanes} in the fibres of $E$.  Thus $\pi_* \mathcal O_{\mathbb P(E)}(r) = S^r E$ for $r\ge 0$.

\section{Preliminaries}

Before discussing almost Hermitian Einstein metrics we recall some basic definitions.  Let $(B,L)$ be a polarised manifold with $b=\dim B$ and $E\to B$ a vector bundle.  We say that $E$ is \emph{Mumford stable} if for all proper coherent subsheaves $F\subset E$ $$ \mu(F)< \mu(E)$$
where the slope $\mu(F)=\mu_L(F)=\deg_L F/\rk(F)$ is the quotient of the degree of $F$ (with respect to $L$) by its rank $\rk(F)$.  We say it is \emph{Mumford semistable} if the same condition holds but with non-strict inequality.  Finally $E$ is \emph{Mumford polystable} if it is the direct sum of Mumford stable bundles whose factors all have the same slope.  

Any Mumford semi-stable bundle $E$ has a Jordan-H\"older filtration $0= F_0\subset F_1 \subset \cdots \subset F_n=E$ by torsion free subsheaves such that the quotients $F_i/F_{i+1}$ are Mumford stable with $\mu(E)=\mu(F_i/F_{i+1})$.  We say that $E$ has a \emph{Jordan-H\"older filtration given by subbundles} if it is Mumford semistable and all the quotients $F_i/F_{i+1}$ are locally free (see \cite[Theorem 3]{L}). \smallskip 

We say that $E$ is \emph{Gieseker stable} if for all proper coherent subsheaves $F\subset E$ one has the following inequality for the normalised Hilbert polynomials
$$\frac{\chi(F\otimes L^k)}{\rk(F)} < \frac{\chi(E\otimes L^k)}{\rk(E)} \quad \text{ for } k\gg 0,$$
and \emph{Gieseker semistability, Gieseker polystability} is defined analogously.  It is known that if $E$ is Gieseker stable then it is simple \cite{K}, which means that $\Ker(\bar{\partial})=\Ker(\partial)=\mathbb{C}\Id_E$ \cite{K,L-T}.

These stability notions are related; using that $\mu_L(F)$ is the leading order term in $k$ of $\chi(F\otimes L^k)/\rk(F)$ one sees immediately that 

 \vspace{-3mm}\begin{center}
  \parbox[c]{2cm}{\vspace{-3mm}\center{ Mumford stable}} $\Rightarrow$ \parbox[c]{2cm}{\vspace{-3mm} \center{Gieseker stable}} $\Rightarrow$ \parbox[c]{2cm}{\vspace{-3mm}\center{Gieseker semistable}} $\Rightarrow$ \parbox[c]{2cm}{\vspace{-3mm}\center{Mumford semistable}}.
\end{center}

\subsection{Almost Hermitian-Einstein metrics and Gieseker stability}

Now suppose $L$ is equipped with a smooth hermitian metric $h_L$ with curvature $\omega:=c_1(h_L)>0$.  

\begin{definition}
  We say that a sequence of hermitian metrics $H_k$ on $E$ is \emph{almost Hermitian-Einstein} if for each $r\ge 0$ the curvature $F_{H_k}$ is bounded in the $C^r$-norm uniformly with respect to $k$, and furthermore
\begin{equation}\label{eq1Leung}
[e^ {F_{H_k}+k\omega \Id_E}\Todd(B)]^{(b,b)}=\frac{\chi(E\otimes L^k)}{\rk(E)} \Id_E \frac{\omega^ b}{b!}.
\end{equation}
\end{definition}
In the above the $(b,b)$ indicates taking the top order forms on the left hand since, and $\Todd(B)=1+c_1(B)+\frac{1}{2}(c_1(B)^2+c_2(B))+\cdots$ is the harmonic representative of the  Todd class with respect to $\omega$.  

By a simple rearrangement this condition implies
\begin{equation}\label{expansionRi}
\sqrt{-1}\Lambda_{\omega} F_{H_k} - \mu(E) \Id_E = T_0 + k^{-1} T_1 + \cdots
\end{equation}
where $T_i\in C^{\infty}(\End(E))$ are hermitian quantities depending on $F_{H_k}$ and $\omega$ that are bounded uniformly over $k$ in the $C^r$-norm.   Moreover we can arrange so
\begin{equation}\label{eq:r0}
  T_0 = -\frac{\Scal(\omega)}{2}\Id_E.
\end{equation}
where $\Scal(\omega)$ is the scalar curvature of $\omega$.

\begin{remark}
By the $C^r$-norm above we mean the sum of the supremum norms of the first $r$ derivatives taken using the pointwise operator norm with respect to a background metric on the bundle in question (that should be fixed once and for all).  From now on we shall write $O_{C^r}(k^i)$ to mean a sum of terms bounded in the $C^r$-norm by $Ck^i$ for some constant $C$.  Thus, in the above, $T_i = O_{C^r}(k^0)=O_{C^r}(1)$.
\end{remark}

In \cite{L}, Leung proved a Kobayashi-Hitchin type correspondence for  Gies\-eker stable vector bundles.

\begin{theorem}[Leung]
Assume that the Jordan-H\"older filtration of $E$ is given by subbundles.     Then $E$ is Gieseker stable if and only if $E$ admits a sequence of almost Hermitian-Einstein metrics for $k\gg 0$.
\end{theorem}

For simplicity we package together the following assumption: \\

\hspace{-2cm}\begin{minipage}[t]{13cm}
\begin{tabular}{p{3cm}p{10cm}}
  \begin{equation} \label{H}\tag{$\mathcal{A}$}
  \end{equation}&
  Let $E$ be an  Gieseker stable holomorphic vector bundle of rank $\rk(E)$ whose Jordan-H\"older filtration is given by subbundles. 
\end{tabular}
\end{minipage}


From now on we will also assume that $E$ is not Mumford stable, otherwise our results are direct consequences of \cite{S,W1,W2}.

\subsection{Balanced metrics\label{BalmetricSection}} 

Suppose now in addition to our metric $h_L$ on $L$ we also have a smooth Hermitian metric $H$ on $E$.  These induce a hermitian metric $H\otimes h_L^k$ on $E\otimes L^k$ which determines an $L^2$-inner product on the space of smooth sections $C^{\infty}(E\otimes L^k)$ given by
$$ ||s||_{L^2}^2 = \int_B |s|_{H\otimes h_L^k}^2 \frac{\omega^b}{b!}.$$
Associated to this data there is a projection operator $P_k\colon C^ {\infty}(B,E\otimes L^k ) \rightarrow H^0(B,E\otimes L^ k)$ onto the space of holomorphic sections for each $k$.  The \emph{Bergman kernel} is defined to be the kernel of this operator which satisfies
$$P_k(f)(x)=\int_B B_k(x,y)f(y)\frac{\omega^b_y}{b!} \quad \text{ for all } f\in C^{\infty}(E\otimes L^k),$$ (see \cite[Section 4]{W1}).

We wish to consider the Bergman kernel restricted to the diagonal, which by abuse of notation we write as $B_k(x) = B_k(x,x)$.  Thus $B_k(x)$ lies in $C^{\infty}(\End(E))$ which we shall refer to as the \emph{Bergman endomorphism} for $E\otimes L^k$, which of course depends on the data $(H\otimes h_L^k,\omega^b/b!)$. We denote by $\Met(E)$ the set of smooth hermitian metrics on the bundle $E$.

\begin{definition}
  We say that the metric $H\in \Met(E)$ is \emph{balanced} at level $k$ if the Bergman endomorphism $B_k(H\otimes h_L^k,\omega^b/b!)$ is constant  over the base, i.e.\ $$B_k=\frac{h^0(E\otimes L^k)}{\rk(E)\Vol_L(B)}\Id_E.$$
\end{definition}

The connection with Gieseker stability is furnished by the following result of X.\ Wang \cite{W1}:

\begin{theorem}[Wang]
The bundle $E$ is Gieseker polystable if and only if there exists a sequence of metrics $H_k$ on $E$ such that $H_k$ is balanced at level $k$ for all $k\gg 0$.
\end{theorem}

Taking $E$ to be the trivial bundle, with the trivial metric, gives an important special case.  Here the only metric that can vary is that on the line bundle $L$, and $B_k$ becomes scalar valued.  To emphasise the importance of this case we use separate terminology:

\begin{definition}\label{def:bergmanfunction}
  The \emph{Bergman function} of a metric $h_L$ on a line bundle $L$ with curvature $\omega$ is the restriction to the diagonal of the kernel of the projection operator $P_k \colon C^{\infty}(B,L^k)\to H^0(B,L^k)$.  This scalar valued smooth function shall be denoted by $\rho_k = \rho_k(h_L,\omega^b/b!)$.
\end{definition}

We say the metric $h_L$ is \emph{balanced at level $k$} if $\rho_k$ is a constant function, i.e.\ 
$$\rho_k = \frac{h^0(L^k)} {\Vol_L(B)}.$$

In this context, balanced metrics are related to stability of the base $B$ as in the following result proved by Zhang \cite{Zh}, Luo \cite{Luo}, Paul \cite{Paul} and Phong-Sturm \cite{PhongSturm}. We refer to the survey \cite{Fut} for recent progress on the notion of asymptotic Chow stability.

\begin{theorem}
  There exists a hermitian metric $h_L$ on $L$ that is balanced at level $k$ if and only if the embedding of $X$ into projective space via the linear series determined by $L^k$ is Chow polystable.
\end{theorem}

\begin{remark}
  It will turn out that the assumption that $\Scal(\omega)$ is constant is not strictly speaking necessary for our proof of Theorem \ref{thm:mainstable}.  In fact, as will be apparent, a simple modification shows it is sufficient to assume that there is a sequence of metrics $h_{L,k}$ on $L$ that are balanced at level $k$ and whose associated curvatures $\omega_k$ are themselves bounded in the right topology.  However we know of no examples of manifolds that admit such a sequence of metrics that do not admit a cscK metric, and thus this generalisation does not give anything new.
\end{remark}

\subsection{Density of States Expansion}

Through work of  S.T.\ Yau \cite{Yau2}, G.\ Tian \cite{Ti1}, D.\ Catlin \cite{Ca}, S.\ Zelditch \cite{Ze},  X.\ Wang \cite{W2} among others, one can understand the behaviour of the Bergman endomorphism as $k$ tends to infinity through the so-called ``density of states'' asymptotic expansion. We refer to \cite{M-M} as a reference for this topic. The upshot is that for fixed $q,r\ge 0$ one can write
\begin{equation}\label{bergmanexpan} 
B_k = k^b A_0 + k^{b-1} A_1 + \cdots + k^{b-q} A_q + O_{C^r}(k^{b-q-1})
 \end{equation}
where $A_i\in C^{\infty}(\End(E))$ are hermitian endormorphism valued functions.  The  $A_i$ depend on the curvature of the metrics in question, and when necessary will be denoted by $A_i = A_i(h,H)$;  in fact
$$ A_0 = \Id_E \quad \text{and} \quad A_1 = {\sqrt{-1}} \Lambda_{\omega} F_{H} + \frac{\Scal(\omega)}{2} \Id_E.$$ 

Now a key point for our application is the observation that the above expansion still holds if the metrics on $L$ and $E$ are allowed to vary, so long as the \emph{curvature} of the metric on $E$ remains under control.  This is made precise in the following proposition which is a slight generalisation of \cite[Theorem 4.1.1]{M-M}.

\begin{proposition}\label{unifexpanB}
Let $q,r\geq 0$ fixed as above. Let $h_{L,k}\in \Met(L)$ be a sequence of metrics  converging in $C^\infty$ topology to $h_L\in \Met(L)$ such that $\omega :=c_1(h_L)>0$. Let $H_k\in \Met(E)$ be a sequence of metrics  such that  the curvatures $F_{H_k}$ are bounded independently of $k$ in $C^{r'}$ norm for some $r'\gg r,q$.   Consider  the Bergman endomorphism $B_k$ associated to $H_k\otimes h_{L,k}^k\in \Met(E\otimes L^k)$. Then  $B_k$ satisfies the uniform asymptotic expansion \eqref{bergmanexpan} with $A_i=A_i(h_{L,k},H_k)$.
\end{proposition}
We only provide a sketch of the proof of this proposition by pointing out how to adapt  Tian's construction of peak sections (\cite{Ti1}, \cite[Section 5]{W2}) to this setting.  It is important for our application that we do not assume the $H_k$ necessarily converge. 

In order to modify a smooth peaked section to a holomorphic one, and control the $L^2$ norm of this change, one needs to apply H\"ormander $L^2$ estimates for the $\bar{\partial}$ operator. But under our assumptions $\i\Lambda F_{H_k}+k  \Id_E$ is positive definite
 for $k$ large enough so H\"ormander's theorem (see  \cite[Theorem (8.4)]{De}) holds on $E\otimes L^k$.

Since the calculation of the asymptotics is local in nature, another key ingredient is a pointwise expansion of the involved metrics. Fix a point $z_0\in B$. From  \cite[Chapter V - Theorem 12.10]{De2}, we know that 
 there exists a holomorphic frame $\left( e_{i}\right) _{i=1,..,rk(E)}$ over a neighbourhood of $z_{0}\in B$ 
such that, with respect to this frame, the endomorphism $ \mathbf{H}_{{k}}(z)_{ij}=H_k(e_i,e_j)$ associated to the metric $H_{k}$ has the following expansion:
\begin{equation}
 \mathbf{H}_{k}(z) _{ij}=\left( \delta _{ij}-\sum_{1\leq
k,l\leq n}\left(F_{H_{k}}\right)_{i\overline{j}k\overline{l}} z_{k}\bar{z}_{l}+\mathbf{O}%
\left( |z|^{3}\right) \right),  \label{20}
\end{equation}
Furthermore, by induction one can show that the higher order terms of the expansions are given by derivatives of the curvature of the metric on $E$. For instance, at order $3$, $\mathbf{H_{k}}(z)_{ij}$ has an extra term of the form $$-\frac{1}{2}\left((F_{H_k})_{i\bar ja\bar b,c} z_a \bar z_b z_c + (F_{H_k})_{i\bar ja\bar b,\bar c} \right) z_a \bar z_b \bar z_c,$$ and thus in \eqref{20}, $\mathbf{O}\left( |z|^{3}\right)=O_{C^{r'-1}}(k^0)$ under our assumptions. Similarly,  the higher order terms of this Taylor expansion are under control. Using this, one can follow line by line the arguments of \cite[Section 5]{W2} to obtain the proposition.

\section{Construction of balanced metrics}

In this section  we construct metrics for $\PP(E)$ that are almost balanced by perturbing the metrics on the bundle $E$.  Then an application of the implicit function argument from \cite{S} will provide the required balanced metrics.  

\subsection{Relating the metric on the bundle to the metric on the projectivisation\label{relate}}

We first recall the techniques in \cite{S} that relate the Bergman endomorphism on $E\otimes L^k$ and the Bergman function on the projectivisation $(\mathbb{P}(E),\mathcal{L}_k:=\mathcal{O}_{\mathbb{P}E}(1)\otimes \pi^* L^{k})$. \smallskip

Let $V$ be a vector space equipped with a hermitian metric $H_V$.  This induces in a natural way a Fubini-Study hermitian metric on $ \mathcal{O}_{\mathbb{P}(V)}(1))$ which we denote by $\hat{h}_V$.  Similarly given a hermitian metric $H$ on $E$ we get an induced metric $\hat{h}_E$ on $\mathcal O_{\mathbb P(E)}(1)$.    We denote by 
$$\rho_k=\rho_k(\hat{h}_E\otimes \pi^*h_L^k)$$
 the Bergman function on $(\PP(E),\mathcal L_k)$ induced from the metric $\hat{h}_E\otimes \pi^*h_L^k$.   The next results gives an asymptotic expansion for $\rho_k$ in  $k$ (observe this is not the same as the usual density of states expansion, since we are not taking powers of a fixed line bundle).

\begin{theorem}[Seyyedali \cite{S}]\label{thm:seyyedali}
There exists smooth endomorphism valued functions $\tilde{B}_k=\tilde{B}_k(H, h_L)$ such that
\begin{equation}\rho_{k}([v])=\frac{1}{c_r} \Tr\left(\frac{{v\otimes v^{*_{H}}}}{\Vert v\Vert_{H}^2}\widetilde{B}_k(H, h_L)\right)\quad\text{for } [v]\in \PP(E),\label{bfunction}\end{equation}
where  $c_r:= \int_{\mathbb{C}^{r-1}}\frac{d\zeta \wedge d\bar\zeta}{(1+\sum_{j=1}^{r-1}|\zeta_j|^2)^{r+1}}$.  Moreover $\tilde{B}_k$ has an asymptotic expansion of the form
\begin{eqnarray}
 \widetilde{B}_k(H, h_L) &=& k^b\Id_E  \label{expansiontildeB} + k^{b-1}\left(\sqrt{-1}[\Lambda_\omega F_{H}]^0 +\frac{\rk(E)+1}{2\rk(E)}\Scal(\omega)\Id_E\right)+\cdots\nonumber
\end{eqnarray}
where $[T]^0$ denotes the traceless part of the operator $T$.
\end{theorem}

We refer to $\tilde{B}_k$ as the \emph{distorted Bergman endomorphism}.  The proof of the previous results is obtained by relating $\tilde{B}_k$ to the Bergman endomorphism by an identity of the form
$$(\sum_{j=0}^b k^{j-b}\Psi_j)\widetilde{B}_k(H, h_L)=B_k(H\otimes h_L^k,\omega^b/b!),$$
for certain $(\Psi_j)_{j=0..b}\in \End(E)$ that depend only on the curvature of the metric $H\in \Met(E)$.  In fact,
\begin{eqnarray*}
\Psi_j&=&\Lambda^{b-j}_\omega\left(F_{H}^{b-j}+P_1(H)F_{H}^{b-j-1}+...+P_{b-j}(H)\right),
 \end{eqnarray*}
where  $P_i(H)=P_i(C_1(H),...,C_{b-j}(H))$ are polynomials of degree $i$ in the $k$-th Chern forms   $C_k(H)$ of $H$, $1\leq k\leq b-j$ \cite[p 594]{S}  \smallskip  

Given this, the asymptotic expansion for $\tilde{B}_k$ follows from that of $B_k$.  Thus, using Proposition \ref{unifexpanB} we see that Theorem \ref{thm:seyyedali} in fact holds uniformly if the metric $h$ is allowed to vary in a compact set, and the metric $H$ is allowed to vary in such a way that the curvature $F_H$ is bounded (as in the case for almost Hermitian-Einstein metrics).

\subsection{Perturbation Argument}
From now on let $E\rightarrow B$ be a vector bundle satisfying assumption (\ref{H}), equipped with a family of almost Hermitian-Einstein metrics $H_k\in \Met(E)$,  and $(L,h_L)$ a polarisation of the underlying manifold $B$ such that $\omega=c_1(h_L)$ is a cscK metric.   We now show how to adapt the methods of \cite[Theorem 1.2]{S}, and prove the existence of  metrics on $\mathcal{L}_k$ that are almost balanced, in the sense that the associated Bergman function is constant up to terms that are negligible for large $k$ \cite{D1}.

The approach is to perturb both the K\"ahler metric $\omega$ and the almost Hermitian-Einstein metric $H_k$ on $E$  by considering
\begin{eqnarray*}
 \omega'_k&=&\omega+ \ddbar(\sum_{i=1}^qk^{-i}\phi_{i}),\\ 
H_k'&=&H_k\left(\Id_E+\sum_{i=1}^q k^{-i}\Phi_{i}\right),
\end{eqnarray*}
where $\phi_i$ are smooth functions on $B$ and $\Phi_i$ are smooth endomorphisms of $E$. We will also denote the perturbed metric on $L$ as $h_L'=h_Le^{-\sum_{i=1}^q k^{-i}\phi_i}\in \Met(L)$ which satisfies $\omega'_k=c_1(h_L')$. 
The perturbations terms $\Phi_i$ and $\phi_i$ will be constructed iteratively to make the distorted Bergman endomorphism approximately constant.   In fact it will be necessary for $\phi_i$ and $\Phi_i$ to themselves depend on $k$, but for fixed $i$ they will be of order $O_{C^r}(k^0)$, and this will be clear from their construction.\smallskip

Observe the metrics $\omega'_k$ lie in a compact set, and the curvature of the metrics $H_k'$ are bounded over $k$.  For the first step of the iteration, where $q=1$, we can apply Theorem \ref{thm:seyyedali} to deduce
\begin{equation}\label{modif}\tilde{B}_k(H_k',h_L')=k^b \Id_E + k^{b-1}A_1(H_k,\omega)+ k^{b-2}(A_2(H_k,\omega)+\delta)+ \cdots\end{equation}
where
$$A_1(H_k,\omega) = \sqrt{-1}[\Lambda_\omega F_{H_k}]^0 +\frac{\rk(E)+1}{2\rk(E)}\Scal(\omega)\Id_E$$ 
and we have defined
$$\delta=\delta(\Phi_1,\phi_1):=D(A_1)_{H_k,\omega}(\Phi_1,\phi_1),$$
where $D(A_1)$ is the linearisation of $A_1$.  Thus
\begin{eqnarray*}
D(A_1)_{H,\omega}(\Phi,\phi)&=&\frac{d}{dt}\Big\vert_{t=0}A_1(H(\Id_E+t\Phi),\omega+t\ddbar\phi) \hspace{-6.5cm} \nonumber \\ 
&=& \frac{\rk(E)+1}{2\rk(E)}(\mathbb{L}\phi)\Id_E \nonumber \\
&&+\i\left[\Lambda_\omega\bar\partial\partial \Phi + \Lambda^2_\omega(F_{H}\wedge \ddbar \phi) - \Delta_{\omega}\phi \Lambda_\omega F_{H}\right]^0 \label{linear2}
\end{eqnarray*}
where $\mathbb L$ denotes the Lichnerowicz operator with respect to $\omega$.

Now from the definition of the almost Hermitian-Einstein metrics we have an expansion \eqref{expansionRi}
\begin{equation}\label{expanT}T:=\sqrt{-1} \Lambda_{\omega}F_{H_{k}} - \mu( E)\Id_{E}=T_0 + T_1k^{-1} + T_2 k ^{-2}+ \cdots + T_{b-1}k^{b-1},
\end{equation}
where the $T_i=O_{C^r}(k^0)$ and $T_0$ is constant, see \eqref{eq:r0}. \\
The metric $H_k$ on $E$ induces a metric on $\End(E)$ (that we still denote by $H_k$ in the sequel) and one has an operator $\partial_{\End(E)}$ (that we still denote $\partial$ in the sequel) as the $(1,0)$ part of the connection operator induced on $\End(E)$ from the Chern connection on $E$ compatible with $H_k$. One can write the associated curvature on $\End(E)$ as $F_{\End(E),H_k}=F_{H_k} \otimes Id_{E^*}+Id_{E}\otimes F_{E^*,H_k^*}$. Thus, we obtain a similar expansion as \eqref{expanT},
\begin{eqnarray}\label{expanR}R&:=&\sqrt{-1} \Lambda_{\omega}F_{\End(E), H_{k}} - \mu(\End(E))\Id_{\End(E)}\\
&=& \sqrt{-1} \Lambda_{\omega}F_{\End(E), H_{k}} \nonumber \\
&=&R_0 + R_1k^{-1} + R_2 k ^{-2}+ \cdots + R_{b-1}k^{b-1}, \nonumber
\end{eqnarray}
where the $R_i=O_{C^r}(k^0)$ and $R_0$ is constant.\\
Using this, we rewrite the distorted Bergman endomorphism as
\begin{eqnarray*}
 \tilde{B}_k(H_k',h_L') &=& k^b \Id_E + k^{b-1} \tilde{A}_1 + k^{b-2} \tilde{A}_2 + \cdots
\end{eqnarray*}
where 
\begin{eqnarray*}
\tilde{A}_1 &=&  \frac{\rk(E)+1}{2\rk(E)}\Scal(\omega)\Id_E\\
\tilde{A}_2 &=&  [T_1]^0 + A_2 + \delta(\Phi_1,\phi_1)
\end{eqnarray*}
since $[T_0]^0=0$.

Observe since $\Scal(\omega)$ is constant, the top coefficient $\tilde{A}_1$ is also constant.   The aim now is to find a perturbation that makes the lower order terms also constant.   To this end it is convenient to define
$$\tilde{R} = R-R_0=O_{C^r}(1/k)$$ 
and to rewrite the Bergman endomorphism once again, this time  in the following way using \eqref{expanR}
\begin{equation}\label{modifiedasympt} \tilde{B}_k(H_k',h_L') = k^b \Id_E + k^{b-1} \tilde{A}_1 + k^{b-2} (\tilde{A}_2 - [\tilde{R}\Phi_1]^0) + k^{b-3} (\tilde{A}_3+ [R_1 \Phi_1]^0) +
 \cdots.\end{equation}
 Remark that since $\Phi_1$ is $O_{C^r}(k^0)$, the same will be true of $[R_1\Phi_1]^0$.    The reason for adding and subtracting this term arises when it comes to ensuring that the $\Phi_i$ we construct are hermitian operators, as in the next proposition which ensures that it is possible to find $\phi_1$ and $\Phi_1$ to make the $k^{b-2}$ term constant. \smallskip

Define $\End_0(E)$ to be the vector space of endomorphisms $\eta$ of $E$ such that $\int_B  \Tr \eta \,\frac{\omega^b}{b!}=0$, $\End_0^0(E)$ the trace free elements of $\End_0(E)$ and $C_0^\infty(B,\mathbb{R})$ the space of smooth functions with null integral with respect to the volume form $\dvol$.

\begin{proposition}\label{construct}
 Assume that $\Aut(B,L)/\mathbb{C}^*$ is discrete, $\omega$ is a cscK metric and that $E$ satisfies assumption (\ref{H}). Then  for any endomorphism $\zeta\in \End_0(E)$ there exists a unique couple $(\Phi_1,\phi_1)\in \End_0^0(E)\times C^{\infty}_0(B,\mathbb{R})$ such that
\begin{equation}
\delta(\Phi_1,\phi_1) -[{\tilde{R}}\Phi_1]^0=\zeta. \label{DA1eqn}
\end{equation}
 Furthermore,  $\Phi_1$ is hermitian with respect to $H_k$ if and only if the same is true of $\zeta$. 

Finally if $r\ge 4$ and $\alpha\in (0,1)$ there is a  $c_{r,\alpha}$ such that for all $\zeta$,
$$||\phi_1||_{C^{r,\alpha}}+ || \Phi_1 ||_{C^{r-2,\alpha}} \le c_{r,\alpha} || \zeta ||_{C^{r-4,\alpha}}.$$
 \end{proposition}

We observe that $\tilde{R}$ is non-zero for all $k$ since by assumption $E$ is not Mumford stable and thus none of the $H_k$ are Hermitian-Einstein.  The proof of the previous proposition will depend on a number of Lemmas, the first of which is a consequence of K\"ahler identities. 
\begin{lemma}\label{kid1}
Let $H$ be a hermitian metric on $E$ which induces a metric on $\End(E)$ that we still denote $H$. Then for any $\zeta \in \End({E})$,  
$$\sqrt{-1}\Lambda_\omega \bar{\partial}\partial \zeta^{*_H} = (\L\zeta - [\sqrt{-1}\Lambda_\omega F_{\End(E),H}, \zeta])^{*_H}.$$
\end{lemma}

\begin{lemma}[Poincar\'e type inequality]\label{Poincare}
 Assume that $E$ is a simple holomorphic vector bundle. Then there is a constant $C$ such that if $H\in \Met(E)$ and $\eta\in \End(E)$, we have
 the following inequality with respect to the metric induced on $\End(E)$, 
$$\Vert \eta\Vert^2_{L^2_H} \leq C\Vert \bar \partial \eta\Vert^2_{L^2_H} + \frac{1}{rk(E)\Vol_L(B)}\Big\vert \int_B \Tr \eta \,\frac{\omega^b}{b!} \Big\vert^2$$
Note that if we consider another reference metric $H_0$ and $H$ such that $r\cdot H_0 > H > r^{-1}\cdot H_0$ with $r>1$, then we can choose $C$ depending only on $(H_0,r)$. 
\end{lemma}
\begin{proof}
 This is standard from the fact $\bar\partial^* \bar\partial$ provides a positive elliptic operator and our simpleness assumption \cite[Section 3]{W2}. Here the constant $C$ in the statement can be taken as the first positive eigenvalue of the elliptic operator. Note that for a varying metric in a bounded family of $\Met(E)$, since the $\bar\partial$-operator doesn't depend on the metric, we can choose the constant $C$ uniformly.
\end{proof}

\begin{lemma}\label{adjoint}
Assume that $E$ is a simple holomorphic vector bundle. For  $k$ sufficiently large, given any $\zeta\in \End_0(E)$ there is a unique $\eta\in \End_0(E)$ such that \begin{equation}\label{tosolve}\L \eta - [\tilde{R} \eta]^0 = \zeta\end{equation}
Furthermore $\eta$ is hermitian (with respect to $H_k$) if and only if the same is true for $\zeta$.  Finally if $r\ge 2$ and $\alpha\in (0,1)$ there is a constant $c_{r,\alpha}$ such that $$||\eta||_{C^{r,\alpha}} \le c_{r,\alpha} ||\zeta||_{C^{r-2,\alpha}}.$$ 
\end{lemma}
\begin{proof}

 We use Fredholm alternative for elliptic equations. Firstly, $\tilde{{R}}$ is hermitian thus, 
the operator $$\eta \mapsto \L\eta -[ \tilde{{R}}\eta]^0$$ is hermitian and elliptic. To show existence of a solution, we need to show that this operator restricted on $\End_0(E)$ has trivial kernel.  Let us assume that 
\begin{equation}
\L \eta -[ \tilde{{R}} \eta]^0 = 0.\label{homogeneous}
 \end{equation}
Let us fix a smooth hermitian metric on $E$ which gives us a metric on $\End(E)$. Equation (\ref{homogeneous}) implies, by K\"ahler identities and by taking inner product with $\eta$, that we have pointwise
$$\langle\partial \eta , \partial \eta\rangle -\langle[\tilde{{R}}\eta]^0,\eta\rangle=0.$$ 
Now, this implies \begin{equation}\label{ded1} \langle\bar{\partial} \eta^*,\bar{\partial} \eta^*\rangle - \langle[\tilde{\underline{R}}\eta]^0, \eta\rangle=\Vert \bar\partial \eta^*\Vert^2 - \langle [\tilde{{R}}\eta]^0,\eta\rangle= 0.\end{equation} Using
Cauchy-Schwartz inequality, we have $\langle\tilde{{R}}\eta, \eta\rangle \leq \Vert \tilde{{R}}\Vert \Vert \eta \Vert^2\leq \Vert \tilde{{R}}\Vert \Vert \eta^* \Vert^2$ and $\langle \mathrm{Tr}(\tilde{{R}}\eta)\Id,\eta\rangle \leq \rk(E)\Vert \tilde{{R}}\Vert \Vert \eta\Vert^2$. By integration, we deduce, using that $\Vert \tilde{{R}} \Vert_{C^0}=O_{C^r}(1/k)$  and Lemma \ref{Poincare}, that $\Vert\bar\partial \eta^*\Vert^2_{L^2}(1-C/k)=0$ from (\ref{ded1}). Thus $\bar\partial \eta^*=0$ if $k\gg 0$. But, since $E$ is simple, this gives that $\eta= \alpha \Id_E$ for a constant $\alpha$, see \cite[Section 7.2]{L-T}. Finally, the kernel of the operator $\L\cdot - [\tilde{{R}}\cdot]$ on $\End_0(E)$ is trivial and we get uniqueness. 

Let us show that we get a hermitian solution. From Lemma \ref{kid1}, one has that 
$$\L \eta^{*}= \left(\L \eta -[\sqrt{-1}\Lambda_\omega F_{\End(E),H_{k}},\eta ]  \right)^{*}$$
where now the adjoint is computed with respect to the almost Hermitian-Einstein  metric $H_k$  on  $E$. Since $[R_0,\eta]=0$ (any term of the form $\theta \Id_{End(E)}$ with $\theta$ a function is in the centre of the Lie algebra $\End(E\otimes L^{k})$), one can rewrite this equation as
$$\L \eta^{*}= \left(\L \eta -[\tilde{R},\eta ]  \right)^{*}.$$
After expansion, this is equivalent to 
$$(\L-\tilde{R}) \eta^{*} = \left((\L-\tilde{R}) \eta\right)^{*}$$
since $\tilde{R}$ is hermitian, and this can be rewritten as
$$(\L\eta^* -[\tilde{R} \eta^{*}]^0) =\left(\L\eta-[\tilde{R}\eta]^0\right)^{*}.$$
Now, from the uniqueness we have shown previously, one gets that the solution is hermitian with respect to the metric $H_k$. 

Let us denote $\End_0(E)^{r,\alpha}$ the  Sobolev space of
$C^{r,\alpha}$ hermitian endomorphisms of $\End_0(E)$. 
For $k\gg 0$, $r\geq 2$, we have that  $\L\cdot - [\tilde{{R}}\cdot]^0$ is an invertible linear differential operator of order 2 from $\End_0(E)^{r,\alpha}$ to $\End_0(E)^{r-2,\alpha}$ with uniformly bounded coefficients since we have the uniform control $\tilde{{R}}=O_{C^r}(1/k)$. The eigenvalues of $\L\cdot$ are strictly positive, while the eigenvalues of the hermitian  operator (of order 0) $[\tilde{{R}}\cdot]^0$  tend to 0 as $k$ becomes larger. Thus this operator is uniformly elliptic, we can apply Schauder theory of elliptic regularity \cite[Section 7.3]{L-T}. Note that we could also invoke the work of Uhlenbeck and Yau for the  operator $\L\cdot$ with a slight generalisation. Finally, the inverse of this operator is bounded and we obtain the existence of a uniform constant $c>0$ such that
for any $(\eta,\zeta)$ satisfying \eqref{tosolve}, $$\Vert \eta\Vert_{C^{r,\alpha}}\leq c \Vert \zeta\Vert_{C^{r-2,\alpha}}.$$
\end{proof}

\begin{proof}[Proof of Proposition \ref{construct}]

Obviously, we have the decomposition $\End_0(E)=\End_0^0(E)\oplus C^{\infty}_0(B,\mathbb{R})\Id_E$.  First we deal with existence, by looking at the kernel of the operator on $\End_0(E)$ given by
$$D(A_1)_{H_k,\omega}(\Phi_1,\phi_1) -[{\tilde{R}}\Phi_1]^0=0$$
where $\Phi_1\in \End_0^0(E)$ and $\phi_1\in C^{\infty}_0(B,\mathbb{R})$. This is equivalent to ask that
\begin{align}
 \frac{\rk(E)}{2\rk(E)+1}\mathbb{L}\phi_1&=&0 \label{rel1}\\
\left[\sqrt{-1}\left(\Lambda_{\omega}\bar\partial\partial \Phi_1 + \Lambda^2_\omega(F_{H_k}\wedge \ddbar \phi_1) - \Delta_{\omega}\phi_1 \Lambda_\omega F_{H_k}\right) -\tilde{R}\Phi_1\right]^0  &=&0 \label{rel2}
\end{align}
Equation (\ref{rel1}) gives immediately that $\phi_1=0$ since the kernel of the Lichnerowicz operator consists of just the constant functions (see \cite{D1}) thanks to the fact that  $\Aut(B,L)/\mathbb{C}^*$ is discrete and since $\int_B \phi_1\dvol=0$. Now, since  $\Phi_1$ is trace free, Equation (\ref{rel2}) reduces to 
$$\L \Phi_1 - [\tilde{R}\Phi_1]^0 =0$$
which admits only the trivial solution, from  Lemma \ref{adjoint} ($\tilde{R}\neq 0$ since the vector bundle $E$ is not Mumford stable). Thus, by Fredholm alternative, we can solve 
Equation (\ref{DA1eqn}). 
Moreover, we know that the terms $\frac{\rk(E)}{2\rk(E)+1}\mathbb{L}\phi_1$ and $\sqrt{-1}\Lambda^2_\omega(F_{H_k}\wedge \ddbar \phi) - \sqrt{-1}\Delta_{\omega}\phi_1 \Lambda_\omega F_{H_k}$ are hermitian. Hence, for the solution $\Phi_1$ of (\ref{DA1eqn}), if $\zeta$ is hermitian,  one can rewrite
this equation as $$\L \Phi_1  -[\tilde{R}\Phi_1]^0 =\zeta'$$
where $\zeta'$ is hermitian with respect to $H_k$. Then, applying Lemma \ref{adjoint}, we get that $\Phi_1$ is hermitian. 
Finally the regularity of the solution is a consequence of  Lemma \ref{adjoint} and the fact that the Lichnerowicz operator is a strongly elliptic operator of order 4.
 \end{proof}

Returning now to the construction of the almost balanced metrics, using Proposition \ref{construct}, we obtain $( \Phi_1,\phi_1)$ such that the second term of \eqref{modifiedasympt} satisfies
$$\tilde{A}_2-[\tilde{R}\Phi_1]^0 = C_2 Id_E$$
or equivalently
$$\delta(\Phi_1,\phi_1)-[\tilde{R}\Phi_1]^0=-{A}_2-[T_1]^0+C_2 Id_E$$
where $C_2$ is a topological constant. Note that we have used here the obvious fact that $\int_B \Tr (C_2-A_2) \,\omega^b=0$.\smallskip

For the next step of our iterative process, we perturb the metrics $H_k$ and
$\omega_k$ at the order $q=2$ and try to find $\Phi_2,\phi_2$ such that the third term of \eqref{modifiedasympt} is constant.   Now this third term can be written
$$A_3(H_k,\omega)+\delta(\Phi_2,\phi_2)- [\tilde{R}\Phi_2]^0+ [T_2]^0+[R_1\Phi_1]^0 +b_{1,2}$$
with $b_{1,2}$ obtained from the deformation of $A_2$, and thus depends only on the $(H_k,\Phi_1,\omega,\phi_1)$ computed at the previous step of the iteration.   We then use the same trick as before, introducing the term $[\tilde{R}\Phi_2]^0$ in order to obtain a hermitian solution, and see that $\Phi_2,\phi_2$ need to satisfy
$$\delta(\Phi_2,\phi_2)- [\tilde{R}\Phi_2]^0=C_3 \Id_E-b_{1,2}-A_3(H_k,\omega)-[R_1\Phi_1]^0$$
where $C_3$ is a topological constant.  Now solutions to this equation are guaranteed just as before using Proposition \ref{construct}.  

Repeating this iteration one sees that at each step one is led to solve the equation
$$\delta(\Phi_i,\phi_i)- [\tilde{R}\Phi_i]^0= \zeta_i$$
where $\zeta_i$ is hermitian with respect to $H_k$ and depends on the computations of the previous steps, i.e on the data $(H_k,\omega,\Phi_1,...,\Phi_{i-1},\phi_1,...,\phi_{i-1})$ and $\int_B \Tr \zeta_i \,\omega^b=0$.
Clearly then the metric that we construct with this process is hermitian.   Thus we have the following result:

\begin{theorem}\label{abal}
Let $E$ be a vector bundle that satisfies assumption (\ref{H}) on the  projective manifold $B$ with $\dim_{\mathbb{C}} B=b$, $(L,h_L)$ a polarisation on $B$ with $\omega=c_1(h_L)>0$. Assume that $\Aut(B,L)/mathbb C^*$ is discrete and $\omega$ is a cscK metric. Consider  an almost Hermitian-Einstein metric $H_{k}\in \Met(E)$. Then any fixed integers $q,r>0$, and $k\gg 0$, the metrics $H_{k}$ and $h_L$ can be deformed  to new metrics $H_k'\in \Met(E)$ and $h_L'\in \Met(L)$ such that the distorted Bergman endomorphism $\tilde{B}_k(H_k',h_L')$
satisfies   $$\tilde{B}_k(H_k',h_L') = k^b\Id_E + \epsilon_k \in \End(E)$$
where $\epsilon_k = O_{C^r}(k^{b-q})$.
\end{theorem}

Next consider  $\hat{h}'$ the metric induced on $\mathcal{O}_{\PP(E)}(1)$ from $H'_k\in \Met(E)$. Then using (\ref{bfunction}) gives the following corollary.

\begin{corollary}
 Under the same assumptions as in Theorem \ref{abal}, for any fixed integers $q,r>0$, and $k\gg 0$ each metric $H_{k}$ and $\omega$ can be deformed  to obtain a smooth hermitian metric $H_k'\in \Met(E)$ and a smooth and $h_L'\in \Met(L)$ such that the induced Bergman function $\rho_k(\hat{h}'\otimes\pi^* {h_L'}^k)$ on $\PP(E)$ satisfies   $$\rho_k(\hat{h}'\otimes\pi^* {h_L'}^k) = \hat{C}k^b + \hat{\epsilon}_k \in C^{\infty}(\PP(E),\mathbb{R})$$
where $\hat{C}$ is a topological constant and $\hat{\epsilon}_k = O_{C^r}(k^{b-q})$.

\end{corollary}

\begin{proof}[Proof of Theorem \ref{thm:mainstable}]
The rest of the proof is the same as  \cite[Theorem 1.2]{S} which shows how it is possible to perturb the almost balanced metrics above to obtain balanced metrics.  Observe that all the estimates in sections 2,3 and 4 of \cite{S} only require $E$ to be simple, which is the case since we are assuming it to be Gieseker stable. Note also that  $\mathbb{P}(E)$ has no nontrivial holomorphic vector fields (\cite[Proposition 7.1]{S}) since $E$ is simple.

Finally, the fact that the existence of a balanced metric on $(\mathbb P(E),\mathcal{L}_k)$ implies the stability of the Chow point induced by $(\mathbb P(E),\mathcal{L}_k)$ \cite{Luo,Zh} since there is no nontrivial automorphism, completing the proof.
\end{proof}

\section{Computation of the Futaki invariant \label{comput}}


We turn now to proving the instability result of Theorem \ref{notasymptoticchowstable}.  We refer the reader to \cite{RT} for an overview and of the concepts involved.  What is required is to consider one parameter degenerations (so called ``test configurations'') of our manifold $\PP(E)$ and these can be constructed rather naturally from subbundles.

Suppose that $F$ is a subbundle of $E$ such that $G:=E/F$ is locally free.   This gives rise to a family of bundles $\mathcal E\to X\times \mathbb C\to \mathbb C$ with general fibre $E$ and central fibre $F\oplus G$ over $0\in \mathbb C$.   Moreover $\mathcal E$ admits a $\mathbb C^*$ action that covers the usual action on the base $\mathbb C$, and whose restriction to $F\oplus G$ scales the fibres of $F$ with weight $1$ and acts trivially on $G$.  (One can see this in a number of ways, for instance if $\xi\in H^1(F\otimes G^*)$ represents the extension determined by $E$ then this action takes $\xi$ to zero as $\lambda\in \mathbb C^*$ tends to zero.)

Setting $\mathcal X=\mathbb P(\mathcal E)\to \mathbb C$ and $\tilde{\mathcal L}_k = \mathcal O_{\mathbb P(\mathcal E)}(1) \otimes \pi^* L^k$ we thus have a flat family of polarised varieties with $\mathbb C^*$ action whose general fibre is $(\PP(E),\mathcal L_k)$ (i.e.\ a test-configuration as introduced in \cite{D4}).

The goal is to calculate the sign of a certain numerical invariant $F_1$ called the Futaki invariant (see \cite{D5}).  We use the convention that if $F_1<0$ then $\PP(E)$ is K-unstable, which is known to imply that it is asymptotically Chow unstable \cite[Theorem 3.9]{RT2}.\smallskip

To make the computations more palatable we restrict to the case that $\rank(E)=2$ over a smooth polarised base $(B,L)$
of complex dimension $b\ge 2$, and assume that $F$ and $G$ are locally free (although the computation is essentially the same without this assumption, see \cite[Section 5.4]{RT}).      We denote by $\ch_2$ the second Chern character, so $\ch_2(F) = c_1(F)^2/2$ and $\ch_2(E) = c_1(E)^2/2 - c_2(E)$.  

 We work initially over a base of complex dimension $b$ since this adds no significant difficulties, although the reader may wish to set $b=2$ which will be all that is necessary for our applications.  To ease notation set $\omega = c_1(L)$ and if $\alpha_i\in H^{2d_i}(B)$ with $d_1+\cdots d_r=b$ we write $\alpha_1.\alpha_2\cdots\alpha_r = \int_X \alpha_1\wedge \cdots \wedge \alpha_r$.

\begin{proposition}\label{computFut}
The Futaki invariant of the test configuration $(\mathcal X,\bar{\mathcal L}_k)$ is\footnote{This corrects an error in the lower order term of \cite[Prop. 5.23] {RT}}
  \begin{equation}
F_1 = C_1 k^{2b-1} +  C_2 k^{2b-2} + O(k^{2b-3})
\end{equation}
where 
\begin{eqnarray*}
  C_1 &=& \frac{\omega^b}{6b!(b-1)!} \left(\mu(E) -\mu(F)\right),\\
C_2 &=& \frac{\omega^b}{12b!(b-2)!}(c_1(E)/2 - c_1(F))c_1(B).\omega^{b-2} \\
&&+ \frac{\omega^b}{3b!(b-2)!}(\ch_2(E)/2 -\ch_2(F)).\omega^{b-2} \\
&&+ \frac{1}{12(b-1)!^2}\left(2c_1(E).\omega^{b-1}-c_1(B).\omega^{b-1}) \right)(\mu(E)-\mu(F)).
\end{eqnarray*}
\end{proposition}

\begin{proof}[Proof of Proposition \ref{computFut}]
Recall  $\pi_* \mathcal  L_k^r = S^r E\otimes L^{rk}$ for $r\ge 0$, so from the Riemann-Roch theorem, we get 
\begin{eqnarray*}
 \chi(\mathcal L_k^r) &=& \chi(S^r E\otimes L^{kr}) = \int_B e^{rk\omega} \ch(S^rE) Td(B), \\
&=& \frac{r^b k^b\omega^b}{b!}\rank(S^r E) \\ 
&& + \frac{r^{b-1}k^{b-1}}{(b-1)!}\omega^{b-1}\left( \rank(S^r E)\frac{c_1(B)}{2} + c_1(S^r E) \right)\\
&&+ \frac{r^{b-2}k^{b-2}}{(b-2)!}\omega^{b-2}\left(\rank(S^r E)\Todd_B^{(2)} + \frac{c_1(S^rE).c_1(B)}{2} + \ch_2(S^rE)\right) \\
&&+O(k^{b-3}),
\end{eqnarray*}
where $\Todd_B^{(2)}$ denotes the second Todd class of $B$, and we use the convention that $O(k^{b-3})$ vanishes if $b=2$.  Now, using the splitting principle, it is elementary to check that
\begin{eqnarray*}
 \rank(S^r E) &=& r+1,\\
 c_1(S^r E) &=& r(r+1)c_1(E)/2,\\
 \ch_2(S^rE) &=& r^3[c_1(E)^2/12 + \ch_2(E)/6] + r^2 \ch_2(E)/2 + O(r).
\end{eqnarray*}
Thus for $r\gg 0$,  $$p(r): = h^0(\mathbb P(E),\mathcal{L}_k^r) = a_0r^{b+1} + a_1 r^b + O(r^{b-1}),$$ where
\begin{eqnarray*}
  a_0 &=& \frac{k^b\omega^b}{b!} +\frac{k^{b-1}\omega^{b-1}c_1(E)}{2(b-1)!} + \frac{k^{b-2}\omega^{b-2}}{(b-2)!} \left(\frac{1}{12} c_1(E)^2+ \frac{1}{6} \ch_2(E)\right) \\
     &=& + O(k^{b-3}),\\
  a_1 &=& \frac{k^b\omega^b}{b!}+ \frac{k^{b-1}\omega^{b-1}}{2(b-1)!}. \left(c_1(B) +c_1(E)\right)\\
&&+ \frac{k^{b-2}\omega^{b-2}}{(b-2)!}\left(\frac{\ch_2(E)}{2}+ \frac{c_1(E).c_1(B)}{4} \right) + O(k^{b-3}).
\end{eqnarray*}
Turning to the central fibre $\mathbb P(F\oplus G)$, we have a splitting
\begin{eqnarray*}
  H^0(\mathbb P(F\oplus G), \tilde{\mathcal{L}}_k^r) &=& H^0(B,S^r(F\oplus G)\otimes L^{kr})\\
&=& \bigoplus_{i=0}^r H^0(B, F^{i} \otimes G^{r-i} \otimes L^{kr}),
\end{eqnarray*}
Moreover this is the eigenspace decomposition for the action, with the $i$-th space having weight $i$.  Let $w(r)$ be the sum of the eigenvalues of the action on this vector space, so
\begin{eqnarray*}
  w(r) &=& \sum_{i=0}^r i h^0(B, F^{i} \otimes G^{r-i} \otimes L^{kr}).
\end{eqnarray*}
Now since $\tilde{\mathcal L_k}$ is relatively ample, the higher cohomology groups vanish, and thus pushing forward to $B$ we have that the higher cohomology groups of $F^{i} \otimes G^{r-i} \otimes L^{kr}$ vanish for $r\gg 0$.  Thus from Riemann-Roch again, $h^0(F^{i} \otimes G^{r-i} \otimes L^{kr})$ equals

  \begin{align*}
\frac{k^b r^b \omega^b}{b!}& + \frac{k^{b-1}r^{b-1} \omega^{b-1}}{(b-1)!}\left(\frac{c_1(B)}{2}  +ic_1(F) + (r-i)c_1(G)\right)\\
&+ \frac{k^{b-2}r^{b-2}\omega^{b-2}}{(b-2)!}\left(\frac{(ic_1(F) + (r-i)c_1(G))^2}{2}+ Td_B^{(2)} \right)\\
& +  \frac{k^{b-2}r^{b-2}\omega^{b-2}}{(b-2)!}\left(\frac{c_1(B)(ic_1(F) + (r-i)c_1(G))}{2} \right) + O(r^{b-3}).
\end{align*}

Now an elementary calculation gives  $w(k) = b_0 r^{b+2} + b_1 r^{b+1}+ O(r^b)$, where
\begin{eqnarray*}
  b_0 &=& \frac{k^b\omega^b}{2b!}  + \frac{k^{b-1}\omega^{b-1}c_1(F)}{3(b-1)!} +\frac{k^{b-1}\omega^{b-1}c_1(G)}{6(b-1)!}\\
&&+\frac{k^{b-2}\omega^{b-2}}{2(b-2)!}\left(\frac{c_1(F)^2}{4} + \frac{c_1(F)c_1(G)}{6} + \frac{c_1(G)^2}{12}\right)+ O(k^{b-3}),\\
b_1 &=& \frac{k^b\omega^b}{2b!} +\frac{k^{b-1}\omega^{b-1}c_1(F)}{2(b-1)!}+ \frac{k^{b-1}\omega^{b-1} c_1(B)}{4(b-1)!} \\
&&+ \frac{k^{b-2}\omega^{b-2}c_1(B)}{2(b-2)!}\left(\frac{c_1(F)}{3} + \frac{c_1(G)}{6}\right) + \frac{k^{b-2}\omega^{b-2}c_1(F)^2}{4(b-2)!} + O(k^{b-3}).
\end{eqnarray*}
The definition of the Futaki invariant is $F_1 = b_0a_1- b_1a_0$, and putting this all together gives the result as stated.
\end{proof}

\begin{proposition}
  Suppose $B$ is a surface, and $\chi(F\otimes L^k) = \chi(E\otimes L^k)/2$ for all $k$.    Suppose also that either $c_1(B)=0$ or $\omega=\pm c_1(B)$.  Then $(\PP(E),\mathcal{L}_k)$ is not K-polystable for $k$ sufficiently large.
\end{proposition}
\begin{proof}
  The previous computations can be extended to show that for a surface one can write the Futaki invariant as
$F_1=C_1k^3+ C_2k^2+ C_3k + C_4$
with $C_1,C_2$ given by Proposition \ref{computFut} that vanish and
\begin{eqnarray}
48C_3 &=& \left(8\deg_L E -4c_1(L)c_1(B)\right)\left(\ch_2(E)/2-\ch_2(F)\right)\nonumber \\
&& + 2c_1(E)^2 (\deg_L(E)/2 -\deg_L F) \nonumber \\
&&+ 2\deg_L(F) c_1(E)c_1(B) - 2\deg_L(E) c_1(B)c_1(F)\nonumber\\
144C_4 &=& c_1(E)^2\left[(c_1(E)/2-c_1(F)).c_1(B) + 6(\ch_2(E)/2-\ch_2(F))\right]\nonumber\\
&&-4c_1(E).c_1(B) (\ch_2(E)/2-\ch_2(F)) \nonumber \\
&&+2(c_1(E).c_1(B) \ch_2(F) - c_1(F).c_1(B)\ch_2(E))\label{C_4}
\end{eqnarray}
It is now an easy computation to check that under our assumptions, the terms $C_3$ and $C_4$ vanish.  Also observe that the degeneration used above is not a product test configuration, so $(\PP(E),\mathcal{L}_k)$ is not K-polystable for $k$ large enough.
\end{proof}

\begin{proof}[Proof of Theorem \ref{notasymptoticchowstable}]
Suppose that $E$ is a rank 2 vector bundle that is Gies\-eker stable but not Mumford stable and $\mu(F) = \mu(E)$.  From Proposition \ref{computFut}, the term $C_1$ vanishes and the Futaki invariant of the test configuration associated to $F$ is 
\[F_1=\frac{k^2}{24}\left( 4\, \left( \ch_2(E)/2-\ch_2(F)^2\right) + c_{1}(B)\left( c_1(E)/2-c_{1}(F)\right)\right) + O(k).\]
  Thus by hypothesis, $F_1<0$ for $k\gg 0$, proving that $(\mathbb P(E),\mathcal L_k)$ is not K-semistable for $k\gg 0$ as claimed.
\end{proof}

\section{Examples\label{ex}}

We end by constructing examples of  polarised surfaces $(B,L)$ and vector bundles $E$ over $B$ that satisfy the assumptions of Theorems \ref{thm:mainstable} and  \ref{notasymptoticchowstable}.   To do so we start with a base $B$ with trivial automorphism group that has an abundance of cscK metrics.

Fix a rank 2 Mumford stable bundle $V$ over a complex projective curve $C$ of genus $g\ge 2$ and define $B=\PP(V)$.  As is well known, using the Narasimhan-Seshadri Theorem \cite{NS} one can prove there exists a cscK metric in each K\"ahler class of  $B$ (see \cite[1.6]{Fine} for the argument).    Moreover as $V$ is simple and $g\ge 2$, there are no infinitesimal automorphisms of $B$ \cite[Proposition 7.1]{S}.

 We seek a suitable vector bundle $E$ over $B$ which is Gieseker stable and not Mumford stable, and from a simple consideration of the dimension of the relevant moduli spaces it is apparent that such bundle exist. In fact the dimension of the (smooth) moduli space of rank 2 Gieseker stable bundle with fixed Chern class $c_1,c_2$ is (when non empty) $4c_2-c_1+4g-3$ \cite[Corollary 18]{Friedman} and using $g\ge 2$ one see this is strictly larger  than the dimension of the (smooth) moduli space of Mumford stable bundle of same type, which is  $4c_2-c_1+3g-3$ \cite[Proposition 6.9]{Mar}.

Next we fix some notations and describe  the ample cone of $B$. The N\'eron-Severi group of $B$ can be identified with $\mathbb{Z}\times \mathbb{Z}$, with generators the class $\mathfrak{b}$ of $\mathcal{O}_{B}(1)$ and the class $\mathfrak{f}$ of a fibre over $C$. 
We have $\mathfrak{b}^2=\deg(V)$, $\mathfrak{f}^2=0$ and $\mathfrak{b}\cdot \mathfrak{f} = 1$ while the anti-canonical divisor is given by $-K_B=2\mathfrak{b}+2(1-g)\mathfrak{f}$.   To ease the computations we may as well take $\deg V=0$. Then, from \cite[Proposition 3.1]{Ta}, or \cite[Proposition 15]{Friedman}, we know that a class $x\mathfrak{b}+y\mathfrak{f}$ is ample if  $x>0$ and $y>0$. 

Following the ideas of  \cite[Proposition 3.9]{Ta}, consider a rank 2 vector bundle $E_1$ obtained as an extension
\[0\rightarrow \mathcal{O}_B\rightarrow E_1 \rightarrow F_1 \rightarrow 0,\]
where $F_1$ has class $-\mathfrak{b}+(m+1)\mathfrak{f}$ for some large positive $m$. To ensure that we can take such an extension that does not split, we need that $Ext^1(\mathcal{O}_B, F_1^*)=H^1(F_1^*)$ is non trivial. But this follows easily from Riemann-Roch since $\chi(F_1^*)=h^0(F_1^*)-h^1(F_1^*)+h^2(F_1^*)\ge -h^1(F^*)$ and\begin{eqnarray*}
\chi(F_1^*)&=&c_1(F_1^*)^2+\frac{c_1(B)}{2}c_1(F_1^*)+\Todd_2(B),\\
 &=&-2(m+1)+\left(-(m+1)+(1-g)\right)+(1-g),\\
&=&-3(m+1)+2(1-g)<0.
 \end{eqnarray*}

Over $B$, we take the polarisation $L_{m+1}=\mathfrak{b}+(m+1)\mathfrak{f}$ and one checks easily that $$\mu(F_1)=\mu(E_1)=0.$$  

We claim that $E_1$ is in fact Mumford semi-stable. A priori,  we need to check stability with respect to any rank 1 torsion free subsheaf $\mathcal{F}$ of $E_1$ but since we are working with a rank 2 bundle on a surface, $\mathcal{F}^{**}$ is a reflexive rank 1 sheaf on $B$ and thus a line bundle. So $\mathcal{F}=\mathcal{O}(D) \otimes \mathcal{I}$ where $\mathcal{O}(D) $ is a line bundle and $\mathcal{I}$ is an ideal sheaf with $0$-dimensional support, so $c_1(\mathcal{F})=c_1(\mathcal{F}^{**})=c_1(\mathcal{O}(D))$. Since $E_1=E_1^{**}$, it is now clear that it is sufficient to consider stability with respect to subbundles of $E$. But, for any rank 1 subbundle $\mathcal{O}(D)$ of $E_1$, either $\mathcal{O}(D) \hookrightarrow \mathcal{O}$ or $F_1\otimes \mathcal O(-D)$ is effective.  In the first case it is immediate that $\mathcal O(D)$ does not destabilise $E_1$.    In the second case if we write the first Chern class of $\mathcal O(D)$ as $x_D\mathfrak{b}+ y_D\mathfrak{f}$ we see by intersecting with ample line bundles that $x_D\leq -1$ and $y_D\leq m+1$. Hence $\mu(\mathcal{O}(D))\leq \mu(F_1)=\mu(E_1)$ and $E_1$ is Mumford semi-stable with respect to $L_{m+1}$ as claimed.  

In order to construct a Gieseker stable bundle which is not Mumford stable, we tensor the previous extension by a line bundle $F_2$ with first Chern class $c_1(F_2) = -\mathfrak{b} +(g-3-m)\mathfrak{f}$, resulting in a non-trivial extension
\[0\rightarrow F_2\rightarrow E \rightarrow F_1\otimes F_2 \rightarrow 0.\]
Observe that $\mu(F_2) =\mu(E)$ and so $\{0\}\subset F_2\subset E$ is the Jordan-H\"older filtration of the Mumford semistable bundle $E$.

 We claim that $E$ is in fact Gieseker stable.  As before, let $\mathcal F = \mathcal O(D)\otimes \mathcal I$ is a rank 1 torsion free subsheaf of $E$, and taking the double dual $\mathcal O(D)$ is a subbundle of $E$.  Then either $\mathcal{O}(D) \hookrightarrow F_2$ or $F_1\otimes F_2\otimes \mathcal O(-D)$ is effective.  In the first case by writing $c_1(\mathcal O(D))=x_D\mathfrak{b}+y_D\mathfrak{f}$ one checks that if $(x_D,y_D)\neq (-1,g-3-m)$ then 
$\mu(\mathcal{O}(D))<\mu(F_2)=\mu(E)$ while if  $(x_D,y_D)=(-1,g-3-m)$ then $\mu(\mathcal O(D))= \mu(E)$ and \[\frac{\ch_2(E)}{2}-\ch_2(\mathcal{O}(D))+\frac{c_1(B)}{2}\left(\frac{c_1(E)}{2}-c_1(\mathcal{O}(D))\right)=\frac{1}{2}>0.\]
Thus by Riemann-Roch we conclude $\frac{1}{2}\chi(E\otimes L_{m+1}^p)>\chi(\mathcal{O}(D)\otimes L_{m+1}^p)$ for $p\gg 0$ and so $\mathcal O(D)$ does not Gieseker destabilise.  Moreover this inequality only improves if $\mathcal O(D)$ is replaced by $\mathcal F$ since $c_2(\mathcal F)$ is the length of the support of $\mathcal I$ and thus is non-negative.       In the second case, in which $F_1\otimes F_2\otimes \mathcal{O}(-D)$ is effective,  one deduces $x_D\le -2$ and $y_D\le g-2$ with at least one inequality being strict, and so $\mu(\mathcal{O}(D))<\mu(F_1\otimes F_2)=\mu(E)$.  Hence $E$ is Gieseker stable with respect to $L_{m+1}$ as claimed.

So Theorem \ref{thm:mainstable} can be applied in this setting and $(\PP(E),\mathcal{L}_k)$ is Chow stable for $k$ sufficiently large where $\mathcal{L}_k=\mathcal{O}_{\PP(E)}(1)\otimes \pi^* L^k_{m+1}$. To apply Theorem \ref{notasymptoticchowstable}, we compute 
\[ 4 (\ch_2(E)/2- \ch_2(F_2)) + c_{1} (B).\left(c_1(E)/2- c_{1}(F_2)  \right)=-m-g+2<0.\]
 Hence $(\PP(E),\mathcal{L}_k)$ is not K-semistable, thereby proving:

\begin{corollary}
 There exists smooth  polarised manifolds $(X,L)$ such that $(X,L)$ is Chow stable but not asymptotically Chow stable. 
\end{corollary}

\bibliography{noteprojgieseker} 

\end{document}